\documentclass{amsart}

\usepackage{amsmath, amssymb, amsthm}

\newtheorem{theorem}{Theorem}[section]
\newtheorem{lemma}[theorem]{Lemma}
\newtheorem{proposition}[theorem]{Proposition}
\newtheorem{corollary}[theorem]{Corollary}
\newtheorem{claim}[theorem]{Claim}

\theoremstyle{definition}
\newtheorem{definition}[theorem]{Definition}
\newtheorem{remark}[theorem]{Remark}

\newcommand{\cf}{\mathrm{cf}}

\newcommand{\bb}{\mathbb}

\newcommand{\lh}{\mathrm{lh}}

\begin{document}
\title{Pseudo-Prikry sequences}
\author{Chris Lambie-Hanson}
\address{Department of Mathematics, Bar-Ilan University, Ramat Gan 5290002, Israel}
\email{lambiec@macs.biu.ac.il}
\thanks{We would like to thank Spencer Unger, conversations with whom led to the initial results 
of this paper. This work was completed while the author was a Lady Davis Postdoctoral 
Fellow at the Hebrew University of Jerusalem and a Coleman-Soref Postdoctoral Fellow 
at Bar-Ilan University; we would like to thank the Lady Davis Fellowship Trust, the Hebrew University, 
the Israel Science Foundation (grant \#1630/14), and Bar-Ilan University. 
Finally, we would like to thank the anonymous referee for a number of 
helpful corrections and suggestions.}
\maketitle

\begin{abstract}
  We generalize results of Gitik, D\v{z}amonja-Shelah, and Magidor-Sinapova on 
  the existence of pseudo-Prikry sequences, which are sequences that approximate the 
  behavior of the generic objects introduced by Prikry-type forcings, in outer models of set theory. Such sequences play 
  an important role in the study of singular cardinal combinatorics by placing 
  restrictions on the type of behavior that can consistently be obtained in outer models. In addition, we provide 
  results about the existence of diagonal pseudo-Prikry sequences, which 
  approximate the behavior of the generic objects introduced by diagonal Prikry-type forcings. 
  Our proof techniques are substantially different from those of previous results and 
  rely on an analysis of PCF-theoretic objects in the outer model.
\end{abstract}

\section{Introduction} \label{intro_section}

One of the most important methods of obtaining consistency results in cardinal
arithmetic and singular cardinal combinatorics consists of starting with a
regular cardinal and singularizing it by forcing. The earliest example of a
forcing to accomplish this task is due to Prikry and is known as \emph{Prikry
forcing}. A number of related forcing notions have been developed since then;
they are known collectively as \emph{Prikry-type forcings}\footnote{For an introduction
  to Prikry-type forcings, the reader is directed to \cite{gitikhandbook}.}.

If $\kappa$ is a measurable cardinal and $U$ is a normal ultrafilter on
$\kappa$, then the Prikry forcing $\bb{P}_U$ adds an $\omega$-sequence, $\langle
\gamma_n \mid n < \omega \rangle$, known as the \emph{Prikry sequence}, that is
cofinal in $\kappa$ and diagonalizes $U$, i.e., for all $X \in U$ and all
sufficiently large $n < \omega$, we have $\gamma_n \in X$. $\bb{P}_U$ adds no bounded
subsets of $\kappa$ and has the $\kappa^+$-c.c. and thus preserves $\kappa^+$.
Other Prikry-type forcings similarly produce sequences that diagonalize
ground-model ultrafilters or sequences of ultrafilters.

A series of results by Gitik \cite{gitik, gitik_witnessing_seq}, D\v{z}amonja
and Shelah \cite{dzamonja_shelah}, and Magidor and Sinapova
\cite{magidor_sinapova} shows that, in certain abstract settings in which $W$ is
an outer model of $V$ and $\kappa$ is a regular cardinal in $V$ that has been
singularized in $W$, one can find in $W$ a sequence in $\kappa$, often called a
\emph{pseudo-Prikry sequence}, that approximates the behavior of a generic
sequence added by a Prikry-type forcing.  Since there may not be a relevant
normal ultrafilter in $V$, the natural object for the pseudo-Prikry sequence to
diagonalize is the club filter on $\kappa$ or $\mathcal{P}_\kappa(\lambda)$ for
some $\lambda > \kappa$.

In this paper, we revisit and extend these results. In the process, we connect 
them with ideas from Shelah's PCF theory. The following theorem is the
starting point for our investigations.  It is due, in the case in which $\kappa$
is inaccessible in $V$ and a singular cardinal in $W$, to D\v{z}amonja and Shelah
\cite{dzamonja_shelah} and, in its general form, to Magidor and Sinapova
\cite{magidor_sinapova}.

\begin{theorem} \label{dzamonja_shelah_theorem}
  Suppose that:
  \begin{enumerate}
    \item $V$ is an inner model of $W$;
    \item in $V$, $\kappa$ is a regular cardinal;
    \item in $W$, $(\kappa^+)^V$ is a cardinal and $|\kappa| > \cf(\kappa) =: \theta$;
    \item in $V$, $\langle C_\alpha \mid \alpha < \kappa^+ \rangle$ is a sequence of 
      clubs in $\kappa$.
  \end{enumerate}
  Then, in $W$, there is a sequence of ordinals $\langle \gamma_i \mid i < \theta \rangle$ 
  such that, for all $\alpha < (\kappa^+)^V$ and all sufficiently large $i < \theta, ~ 
  \gamma_i \in C_\alpha$. Moreover, for any fixed $\eta < |\kappa|$, we may require that, 
  for all $i < \theta$, $\cf(\gamma_i) > \eta$.
\end{theorem}

A similar result was proved by Gitik \cite{gitik} under the additional
assumption that, in $W$, $|\kappa| > 2^\theta$.  Theorem
\ref{dzamonja_shelah_theorem} has a number of combinatorial applications,
including the following two.

\begin{theorem}[Cummings-Schimmerling, \cite{cummings_schimmerling}] \label{cs_theorem}
  Suppose that:
  \begin{enumerate}
    \item $V$ is an inner model of $W$;
    \item in $V$, $\kappa$ is an inaccessible cardinal;
    \item in $W$, $\kappa$ is a singular cardinal and $\cf(\kappa) = \omega$;
    \item $(\kappa^+)^W = (\kappa^+)^V$.
  \end{enumerate}
  Then, in $W$, $\square_{\kappa, \omega}$ holds.
\end{theorem}

Cummings and Schimmerling prove Theorem \ref{cs_theorem} in the
special case in which $W$ is an extension of $V$ by Prikry forcing at $\kappa$,
but their proof works with a sequence as in Theorem
\ref{dzamonja_shelah_theorem} in place of a true Prikry sequence.

\begin{theorem}[Brodsky-Rinot, \cite{brodsky_rinot_forcing}]
  Suppose that $\lambda$ is a regular, uncountable cardinal, $2^\lambda = \lambda^+$, and $\bb{P}$ is a 
  $\lambda^+$-c.c.\ forcing notion of size $\leq \lambda^+$. Suppose moreover that, in 
  $V^{\bb{P}}$, $\lambda$ is a singular ordinal and $|\lambda| > \cf(\lambda)$. Then, 
  in $V^{\bb{P}}$, there is a $(\lambda^+)^V$-Souslin tree.
\end{theorem}

In \cite{magidor_sinapova}, Magidor and Sinapova generalize Theorem \ref{dzamonja_shelah_theorem} 
to the context of clubs in $\mathcal{P}_\kappa(\kappa^{+m})$ for $m < \omega$. 

\begin{theorem}[Magidor-Sinapova, \cite{magidor_sinapova}] \label{magidor_sinapova_theorem}
  Suppose that:
  \begin{enumerate}
    \item $V$ is an inner model of $W$;
    \item in $V$, $\kappa$ is a regular cardinal;
    \item $m < \omega$ and, in $W$, for all $k \leq m$, $\cf((\kappa^{+k})^V) = \omega$;
    \item in $W$, $(\kappa^{+m+1})^V$ remains a cardinal and $\kappa > \omega_1$;
    \item in $V$, $\langle \mathcal{D}_\alpha \mid \alpha < \kappa^{+m+1} \rangle$ is a 
      sequence of clubs in $\mathcal{P}_\kappa(\kappa^{+m})$.
  \end{enumerate}
  Then, in $W$, there is a sequence $\langle x_i \mid i < \omega \rangle$ such that, for all 
  $\alpha < (\kappa^{+m+1})^V$ and all sufficiently large $i < \omega$, $x_i \in \mathcal{D}_\alpha$.
\end{theorem}

Let us now state the main result of this paper, which generalizes Theorems 
\ref{dzamonja_shelah_theorem} and \ref{magidor_sinapova_theorem}. We first need
the notion of a \emph{fat tree}. Such trees are used
often in the setting of Prikry-type forcing in the proof of a strong form of the
Prikry lemma and a characterization of genericity.  For a tree $T$ of sequences
and an element $\sigma$ of $T$, we write $\mathrm{succ}_T(\sigma)$ for $\{\alpha \mid
\sigma ^\frown \langle \alpha \rangle \in T\}$.

\begin{definition}
  Suppose that $m < \omega$, $\kappa$ is a regular cardinal, and, for $i \leq m$, $\lambda_i \geq \kappa$ is a regular cardinal. 
  Then $T \subseteq \bigcup_{k \leq m+1} \prod_{i < k} \lambda_i$ is a \emph{fat tree of type $(\kappa, \langle \lambda_0, \ldots, 
  \lambda_m \rangle)$} if:
  \begin{enumerate}
    \item for all $\sigma \in T$ and $i < \lh(\sigma), ~ \sigma \restriction i \in T$;
    \item for all $\sigma \in T$, if $\lh(\sigma) \leq m$, then $\mathrm{succ}_T(\sigma)$ is $(<\kappa)$-club in $\lambda_{\lh(\sigma)}$.
  \end{enumerate}
\end{definition}

\begin{remark}
  In the setting of Prikry-type forcing, the definition of a fat tree requires
  that $\mathrm{succ}_T(\sigma)$ is measure-one for some relevant measure
  (cf. \cite[Definition 5.16]{gitikhandbook}). In the abstract setting, in which 
  the existence of such measures is not assumed, the $(<\kappa)$-club filter seems 
  to be the correct analogue of the measure.
\end{remark}

\begin{theorem} \label{fat_tree_guessing_theorem}
  Suppose that:
  \begin{enumerate}
    \item $V$ is an inner model of $W$;
    \item in $V$, $\kappa < \lambda$ are cardinals, with $\kappa$ regular;
    \item $(\lambda^+)^V$ remains a cardinal in $W$;
    \item in $W$, there is a $\subseteq$-directed subset $Y$ of $(\mathcal{P}_\kappa(\lambda))^V$ 
      such that $\bigcup Y = \lambda$ and $|Y|^{+3} < (\lambda^+)^V$;
    \item $m < \omega$ and, in $V$, $\langle \lambda_i \mid i \leq m \rangle$ is a sequence of 
      regular cardinals from the interval $[\kappa, \lambda]$ and $\langle T(\alpha) \mid \alpha < \lambda^+ \rangle$ 
      is a sequence of fat trees of type $(\kappa, \langle \lambda_0, \ldots, \lambda_m \rangle)$.
  \end{enumerate}
  Then, in $W$, there are ordinals $\langle \delta_{i,y} \mid i \leq m, ~ y \in Y \rangle$ such that:
  \begin{enumerate}
    \item[(i)] for all $\alpha < (\lambda^+)^V$, there is $y \in Y$ such that, for all $z \in Y / y 
      := \{z \in Y \mid y \subseteq z\}$, $\langle \delta_{i,z} \mid i \leq m \rangle \in T_\alpha$;
    \item[(ii)] for all regular cardinals $\nu$ such that $\nu^{+2} < (\lambda^+)^V$, there is $y \in Y$ such that, 
      for all $z \in Y / y$ and all $i \leq m$, $\cf(\delta_{i,y}) > \nu$.
  \end{enumerate}
\end{theorem}

Our method of proof is substantially different
from that of the previous results, with a number of advantages. First, we directly 
obtain pseudo-Prikry sequences as in Theorem \ref{dzamonja_shelah_theorem} 
that simultaneously witness the ``moreover'' clause for \emph{every} $\eta < |\kappa|.$\footnote{We thank 
  the referee for pointing out, that, under the mild additional assumption that $2^\theta < \kappa^+$ in $W$, 
  such sequences can also be obtained from Theorem \ref{dzamonja_shelah_theorem}.} Second, we are
able to obtain pseudo-Prikry sequences as in Theorem
\ref{magidor_sinapova_theorem}, but in which the cardinal $\kappa$ and some
cardinals above change their cofinalities to some uncountable cardinal (see
Corollaries \ref{ms_cor} and \ref{magidor_sinapova_cor}).  Third,
we can use these techniques to obtain results about models $V \subseteq W$ such
that $\kappa$ is a regular cardinal in $V$ which is singularized in $W$ and the
least $V$-cardinal above $\kappa$ that is preserved in $W$ is, in $V$, the
successor of a singular cardinal (see Theorems \ref{diag_thm1} and
\ref{diagonal_guessing_theorem}). In this situation, we obtain results about
the existence of certain \emph{diagonal pseudo-Prikry sequences}, 
which approximate the behavior of generic sequences for 
diagonal Prikry-type forcings and certain extender-based forcings.
This allows us to address \cite[Question 1]{magidor_sinapova}, which asks whether Theorem 
\ref{magidor_sinapova_theorem} can be extended by replacing $\kappa^{+m}$ with 
a singular cardinal greater than $\kappa$. Gitik, in \cite{gitik_witnessing_seq}, 
shows that there can be no straightforward generalization of Theorem 
\ref{magidor_sinapova_theorem} to singular cardinals, but our results 
indicate that there is a positive generalization if one replaces psuedo-Prikry 
sequences with diagonal pseudo-Prikry sequences. 

The proofs of our results use PCF-theoretic ideas. 
In particular, given models $V \subseteq W$ of ZFC, we will construct 
combinatorial objects in $V$ and then use PCF-theoretic techniques to 
analyze the properties of these objects in $W$. This analysis will yield 
the pseudo-Prikry sequences that we are seeking.

The structure of the paper is as follows. In Section \ref{prelim_section}, we
give some PCF-theoretic background. In Section \ref{directed_section}, we
present some technical lemmas about the existence of certain nice directed sets in outer models in which
cofinalities have been changed. In Section \ref{guessing_section}, we
describe a connection between fat trees and clubs in
$\mathcal{P}_\kappa(\kappa^{+m})$. We then prove Theorem \ref{fat_tree_guessing_theorem}
and establish some of its notable corollaries. In Section
\ref{diagonal_section}, we use our techniques to obtain results about the
existence of diagonal pseudo-Prikry sequences.

Our notation is, for the most part, standard. If $\lambda < \beta$, where 
$\lambda$ is a cardinal and $\beta$ is an ordinal, then $S^\beta_\lambda 
:= \{\alpha < \beta \mid \cf(\alpha) = \lambda\}$ and 
$S^\beta_{<\lambda} := \{\alpha < \beta \mid \cf(\alpha) < \lambda\}$. A set of ordinals 
$x$ is \emph{$\lambda$-closed} if, whenever $\alpha < \sup(x)$, $\cf(\alpha) = \lambda$, and 
$\sup(x \cap \alpha) = \alpha$, we have $\alpha \in x$. Given a cardinal $\kappa$, 
$x$ is \emph{$(<\kappa)$-closed} if it is $\lambda$-closed for all regular 
$\lambda < \kappa$. If $\beta$ is an ordinal, then we say $x$ is \emph{$\lambda$-club} 
(resp. \emph{$(<\kappa)$-club}) in $\beta$ if it is $\lambda$-closed 
(resp. $(<\kappa)$-closed) and unbounded in $\beta$. If $\sigma$ is a sequence, then 
$\lh(\sigma)$ denotes the length of $\sigma$. If $\kappa < \lambda$ are cardinals, 
$x \in \mathcal{P}_\kappa(\lambda)$, and $A \subseteq \mathcal{P}_\kappa(\lambda)$, 
then $A/x := \{y \in A \mid x \subseteq y\}$.

\section{PCF-theoretic background} \label{prelim_section}

Let $X$ be a set and $I$ be an ideal on $X$.  We write $I^+$ for $\{ A \subseteq
X \mid A \notin I\}$ and $I^*$ for $\{A \subseteq X \mid X \setminus A \in I
\}$.

\begin{definition}
  For $f,g \in {^X}\mathrm{On}$, we define:
  \begin{itemize}
    \item $f <_I g$ if $\{x \in X \mid f(x) \geq g(x)\} \in I$;
    \item $f \leq_I g$ if $\{x \in X \mid f(x) > g(x)\} \in I$;
    \item $f =_I g$ if $\{x \in X \mid f(x) \neq g(x)\} \in I$.
  \end{itemize}
\end{definition}

Unless $I$ is a prime ideal, it is not necessarily the case that 
the conjunction of $f \leq_I g$ and $f \not<_I g$ is equivalent to $f =_I g$.
If $\mu$ is a regular cardinal and $f,g \in {^\mu}\mathrm{On}$, then 
$f <^* g$ denotes $f <_{I_{bd}} g$, where $I_{bd}$ is the ideal of bounded subsets of $\mu$. 
$f \leq^* g$, etc.\ are defined similarly. We write $f < g$ if $f(x) < g(x)$ 
for \emph{all} $x \in X$.

\begin{definition}
  Suppose that $\vec{f} = \langle f_\alpha \mid \alpha < \lambda \rangle$ 
  is a $<_I$-increasing sequence of functions in ${^X}\mathrm{On}$. Suppose also that $g \in {^X}\mathrm{On}$.
  \begin{itemize}
    \item $g$ is a \emph{$<_I$-upper bound} for $\vec{f}$ if, for all $\alpha < \lambda$, $f_\alpha <_I g$.
    \item $g$ is a \emph{$<_I$-exact upper bound} (eub) for $\vec{f}$ if it is a $<_I$-upper bound for $\vec{f}$ and, 
      whenever $h <_I g$, there is $\alpha < \lambda$ such that $h <_I f_\alpha$.
  \end{itemize}
\end{definition}

  If the ideal $I$ is clear from context, then `$<_I$' will be dropped from `$<_I$-upper bound,' etc. 
  It is immediate that, if $g$ and $h$ are both 
  $<_I$-eubs for the same sequence of functions, then $\{x \in X \mid g(x) \neq h(x)\} \in I$.

We now present a standard lemma from \cite{magidor_shelah} regarding cofinalities of exact upper bounds.
The result is stated in \cite{magidor_shelah} in the case that $I$ is the ideal of bounded subsets of a 
regular cardinal, but the proof goes through in the general case.

\begin{lemma}[Magidor-Shelah, \cite{magidor_shelah}, Lemmas 7 and 8] \label{small_cf_lemma}
  Suppose that:
  \begin{itemize}
    \item $\lambda$ is a regular cardinal;
    \item $\vec{f} = \langle f_\alpha \mid \alpha < \lambda \rangle$ is a $<_I$-increasing sequence 
      from ${^X}\mathrm{On}$;
    \item $g$ is an eub for $\vec{f}$.
  \end{itemize}
  Then: 
  \begin{enumerate}
    \item $\{x \in X \mid \cf(g(x)) > \lambda\} \in I$;
    \item if $\delta$ is a regular cardinal and $|X| < \delta < \lambda$, then 
      $\{x \in X \mid \cf(g(x)) = \delta\} \in I$.
  \end{enumerate}
\end{lemma}

\begin{definition}
  Suppose that $\mathcal{F}$ is a collection of functions in ${^X}\mathrm{On}$. Then $\sup(\mathcal{F})$ 
  denotes the function $g \in {^X}\mathrm{On}$ given by letting $g(x) = \sup\{f(x) \mid f \in \mathcal{F}\}$ 
  for all $x \in X$.
\end{definition}

The following lemma concerning the existence of exact upper bounds is essentially due to Shelah and 
will play a key role in our results. For a proof, see 
\cite[Theorem 2.15 and Lemma 2.19]{abraham_magidor}.

\begin{lemma} \label{club_guessing_lemma}
  Suppose $\mu$ and $\lambda$ are regular cardinals such that $|X| < \mu < \mu^{++} < \lambda$, 
  and suppose that $\vec{f} = \langle f_\alpha \mid \alpha < \lambda \rangle$ is a $<_I$-increasing sequence of 
  functions from ${^X}\mathrm{On}$ such that, for every $\gamma \in S^\lambda_{\mu^{++}}$, 
  there is a club $D_\gamma \subseteq \gamma$ such that $\sup\{f_\alpha \mid \alpha \in D_\gamma\} <_I f_\gamma$.
  Then there is an eub $g$ for $\vec{f}$ such that, for all $x \in X$, $\cf(g(x)) > \mu$.
\end{lemma}

We now present one of the central notions of PCF theory.

\begin{definition}
  Suppose $\mu$ is a singular cardinal, $\cf(\mu) = \theta$, and $\langle \mu_\xi \mid \xi < \theta \rangle$ 
  is an increasing sequence of regular cardinals, cofinal in $\mu$. Suppose $\lambda > \mu$ is a regular cardinal. 
  A sequence $\langle f_\alpha \mid \alpha < \lambda \rangle$ is a \emph{scale of length $\lambda$ in $\prod_{\xi < \theta} 
  \mu_\xi$} if it is increasing and cofinal in $(\prod_{\xi < \theta} \mu_\xi, <^*)$.
\end{definition}

\begin{theorem}[Shelah, \cite{cardinal_arithmetic}] \label{scale_thm}
  Suppose that $\mu$ is a singular cardinal and $\cf(\mu) = \theta$. Then there is an increasing sequence $\langle 
  \mu_\xi \mid \xi < \theta \rangle$ of regular cardinals, cofinal in $\mu$, such that there is a scale of length 
  $\mu^+$ in $\prod_{\xi < \theta} \mu_\xi$.
\end{theorem}

We now turn to some basic results concerning elementary substructures. 
Suppose $\kappa$ is a regular, uncountable cardinal and $m < \omega$. Let 
$\Upsilon > \kappa$ be a sufficiently large regular cardinal, and let 
$\vartriangleleft$ be a well-ordering of $H(\Upsilon)$. We will abuse 
notation and write $M \prec H(\Upsilon)$ to mean $M \prec (H(\Upsilon), \in, \vartriangleleft)$. Let 
\[
  \mathcal{X} = \{M \cap \kappa^{+m} \mid M \prec H(\Upsilon), ~  
  |M| < \kappa,\mbox{ and }M \cap \kappa \in \kappa\},
\] 
and note that $\mathcal{X}$ is a 
club in $\mathcal{P}_\kappa(\kappa^{+m})$. For $x \in \mathcal{X}$, let $\kappa_x = x \cap \kappa$, and define 
$\chi_x \in \prod_{i \leq m} \kappa^{+i}$ by $\chi_x(i) = \sup(x \cap \kappa^{+i})$. 
The following lemma is standard; a proof can be found in \cite{magidor_sinapova}.

\begin{lemma} \label{uniqueness_lemma}
  Suppose $x, y \in \mathcal{X}$, $\chi_x = \chi_y =: \chi$, and, for all $i \leq m$, 
  $\cf(\chi(i)) > \omega$. Then $x = y$.
\end{lemma}

We end this section with the following lemma, whose proof, which can be found in \cite{proper_forcing},
provides a simple example of using a combinatorial object defined in an inner model, $V$, to obtain information about an outer 
model, $W$, a technique that we will exploit later in the paper.

\begin{lemma}[Shelah, \cite{proper_forcing}, Chapter XIII, Lemma 4.9] \label{shelah_lemma}
  Suppose that $V$ is an inner model of $W$, $\lambda$ is a regular cardinal in $V$, and $(\lambda^+)^V$ remains a cardinal in 
  $W$. Then, in $W$, $\cf(\lambda) = \cf(|\lambda|)$.
\end{lemma}

\section{Two covering lemmas} \label{directed_section}

In this section, we prove two technical covering lemmas which will be useful for our results in Sections 
\ref{guessing_section} and \ref{diagonal_section}. 

\begin{lemma} \label{first_covering_lemma}
  Suppose that:
  \begin{enumerate}
    \item $V$ is an inner model of $W$;
    \item $\theta < \kappa < \lambda$ are cardinals in $V$, with $\theta$ and $\kappa$ regular and 
      $\lambda \leq \kappa^{+\theta^+}$;
    \item for every $V$-regular cardinal $\nu \in [\kappa, \lambda)$, $\cf^W(\nu) = \theta$. 
  \end{enumerate}
  Then, for every $V$-cardinal $\nu \in [\kappa, \lambda)$, there is, in $W$, a $\subseteq$-directed 
  subset $Y_\nu \subseteq (\mathcal{P}_\kappa(\nu))^V$ such that $|Y_\nu| = \theta$ and $\bigcup Y_\nu = \nu$. 
  Moreover, if $\theta = \omega$, then we can arrange so that $(Y_\nu, \subseteq)$ is isomorphic to $(\theta, <)$.
\end{lemma}

The ``moreover'' part of Lemma \ref{first_covering_lemma} appears as \cite[Propositions 0.6, 0.7]{gitik_witnessing_seq}.

\begin{proof}
  We proceed by induction on $V$-cardinals $\nu \in [\kappa, \lambda)$. For the base case $\nu = \kappa$, 
  simply fix in $W$ an increasing sequence of ordinals $\langle \kappa_i \mid i < \theta \rangle$ cofinal 
  in $\kappa$ and let $Y_\kappa = \{\kappa_i \mid i < \theta\}$. Consider now the successor case. 
  Suppose $\nu \in [\kappa, \lambda)$ and we have constructed $Y_\nu$. We construct 
  $Y_{\nu^+}$. Fix in $V$ a sequence $\langle e_\beta \mid 
  \beta < \nu^+ \rangle$ such that, for all $\beta < \nu^+$, $e_\beta:\nu \rightarrow \beta$ is a surjection.
  In $W$, let $\langle \beta_i \mid i < \theta \rangle$ be increasing and cofinal in $\nu^+$. Let 
  $Z_{\nu^+} = \{e_{\beta_i}``x \mid x \in Y_\nu, ~ i < \theta\}$, and let $Y_{\nu^+}$ consist of all finite unions of 
  elements of $Z_{\nu^+}$. Then $Y_{\nu^+}$ is a $\subseteq$-directed subset of $(\mathcal{P}_\kappa(\nu^+))^V$, 
  $|Y_{\nu^+}| = \theta$, and $\bigcup Y_{\nu^+} = \nu^+$.

  We finally consider the limit case. Suppose $\nu \in (\kappa, \lambda)$ is a limit cardinal in $V$ and we have 
  defined $Y_\mu$ for all $\mu \in (\kappa, \nu)$. Let $Z_\nu = \bigcup_{\mu \in (\kappa, \nu)} Y_\mu$, and 
  let $Y_\nu$ consist of all finite unions of elements of $Z_\nu$. Then $Y_\nu \subseteq (\mathcal{P}_\kappa(\nu))^V$, 
  $\bigcup Y_\nu = \nu$, and, since $\lambda \leq \kappa^{+\theta^+}$ in $V$, $|Y_\nu| = \theta$. 

  To show the ``moreover'' clause, fix a $V$-cardinal $\nu \in [\kappa, \lambda)$ and enumerate 
  $Y_\nu$ as $\{x_k \mid k < \omega\}$. By recursion on $\ell < \omega$, 
  define a subset $Y^*_\nu = \{y_\ell \mid \ell < \omega\}$ of $Y_\nu$ 
  such that, for all $k < \ell < \omega$, we have $x_k, y_k \subseteq y_\ell$. The construction 
  is straightforward, using the directedness of $Y_\nu$. Then $Y^*_\nu \subseteq (\mathcal{P}_\kappa(\nu))^V$, 
  $(Y^*_\nu, \subseteq)$ is isomorphic to $(\omega, <)$, and $\bigcup Y^*_\nu = \nu$.
\end{proof}

\begin{lemma} \label{second_covering_lemma}
  Suppose that:
  \begin{enumerate}
    \item $V$ is an inner model of $W$;
    \item $\theta \leq \kappa$ are regular cardinals in $V$;
    \item $m < \omega$ and, for all $i \leq m$, $\cf^W(\kappa^{+i}) = \theta$.
  \end{enumerate}
  Then, for every $i \leq m$, there is, in $W$, a $\subseteq$-increasing and cofinal sequence 
  $\vec{y}_i = \langle y_{i, \eta} \mid \eta < \theta \rangle$ from $(\mathcal{P}_\kappa((\kappa^{+i})^V))^V$.
\end{lemma}

\begin{proof}
  For $i \leq m$, let $X_i =(\mathcal{P}_\kappa((\kappa^{+i})^V))^V$. 
  We proceed by induction on $i \leq m$. For $i = 0$, the conclusion is immediate, as it is 
  witnessed by any increasing sequence of ordinals $\langle \mu_\eta \mid \eta < \theta \rangle$ 
  cofinal in $\kappa$. Thus, suppose $i < m$ and we have found $\vec{y}_i$. 
  We will construct $\vec{y}_{i+1}$.

  Fix, in $V$, a sequence $\langle e_\beta \mid \kappa^{+i} \leq \beta < \kappa^{+i+1} \rangle$ such that, for all $\kappa^{+i} \leq \beta < \kappa^{+i+1}$, 
  $e_\beta:\kappa^{+i} \rightarrow \beta$ is a bijection. 
  Suppose that, in $W$, $\langle \beta_\xi \mid \xi < \theta \rangle$ is an increasing sequence of ordinals, 
  cofinal in $(\kappa^{+i+1})^V$, such that $\beta_0 \geq (\kappa^{+i})^V$. For $\xi, \eta < \theta$, 
  let $z_{\xi, \eta} = e_{\beta_\xi}``y_{i, \eta}$, and note that, as $e_{\beta_\xi}, y_{i, \eta} \in V$, 
  we have $z_{\xi, \eta} \in X_{i+1}$.

  \begin{claim} \label{covering_claim}
    For all $\xi_0 < \xi_1 < \theta$ and $\eta_0 < \theta$, there is $\eta_1 < \theta$ such that 
    $z_{\xi_0, \eta_0} \subseteq z_{\xi_1, \eta_1}$.
  \end{claim}

  \begin{proof}
    Let $w = (e_{\beta_{\xi_1}}^{-1})``z_{\xi_0, \eta_0}$. Then $w \in X_i$, so,
    since $\vec{y}_i$ is $\subseteq$-cofinal in $X_i$, there is $\eta_1 < \theta$
    such that $w \subseteq y_{i,\eta_1}$.  But then $z_{\xi_0, \eta_0} \subseteq
    z_{\xi_1, \eta_1}$.
  \end{proof}

  We now construct $\langle y_{i+1, \xi} \mid \xi < \theta \rangle$ by recursion on $\xi$ such that, 
  for each $\xi < \theta$, there is $\eta_\xi < \theta$ such that $y_{i+1, \xi} = z_{\xi, \eta_\xi}$. Let 
  $y_{i+1, 0} = z_{0,0}$. If $\xi < \theta$ and we have constructed $\langle y_{i+1, \zeta} \mid \zeta < \xi \rangle$,
  let $\eta^* = \max(\xi, \sup(\{\eta_\zeta \mid \zeta < \xi\}))$, and
  use Claim \ref{covering_claim} to find $\eta_\xi < \theta$ sufficiently large so that, for all $\zeta < \xi$ and 
  $\eta < \eta^*$, $z_{\zeta, \eta} \subseteq z_{\xi, \eta_\xi}$, and let $y_{i+1, \xi} = z_{\xi, \eta_\xi}$. 
  It is easily verified that $\vec{y}_{i+1}$ is $\subseteq$-increasing. To check that it is 
  cofinal, fix $x \in X_{i+1}$. Since $\kappa^{+i+1}$ is regular in $V$, there is $\xi < \theta$ 
  such that $x \subseteq \beta_\xi$. It follows that there is $\eta < \theta$ 
  such that $x \subseteq z_{\xi, \eta}$. But then, for $\xi^* > \max\{\xi, \eta\}$, 
  our construction guarantees $x \subseteq y_{i+1, \xi^*}$.
\end{proof}

\section{Fat trees and outside guessing} \label{guessing_section}

In this section, we prove Theorem \ref{fat_tree_guessing_theorem}. 
We start with the following lemma, which provides some motivation for the consideration of fat trees.

\begin{lemma} \label{fat_tree_model_lemma}
  Suppose that $\kappa$ is a regular, uncountable cardinal, $m < \omega$, and $\mathcal{C}$ is a club in 
  $\mathcal{P}_\kappa(\kappa^{+m})$. For $i \leq m$, let $\lambda_i = \kappa^{+m-i}$. Then there is 
  a fat tree $T$ of type $(\kappa, \langle \lambda_0, \ldots, \lambda_m \rangle)$ such that, 
  for every $\sigma \in T$ with $\lh(\sigma) = m+1$, there is $x \in \mathcal{C}$ such that, 
  for all $i \leq m$, $\sup(x \cap \lambda_i) = \sigma(i)$.
\end{lemma}

\begin{proof}
  Fix a function $F:[\kappa^{+m}]^{<\omega} \rightarrow \mathcal{P}_\kappa(\kappa^{+m})$ such that 
  \[
    \mathcal{C}_F := \{x \in \mathcal{P}_\kappa(\kappa^{+m}) \mid x \cap \kappa \in \kappa\mbox{ and }
    \forall z \in [x]^{<\omega}, ~ F(z) \subseteq x\} \subseteq \mathcal{C}.
  \]
  By recursion on $i \leq m+1$, we will 
  construct $T_i := \{\sigma \in T \mid \lh(\sigma) = i\}$, arranging that, for all $\sigma \in T$ 
  and $k < \lh(\sigma)$, we have $\omega \leq \cf(\sigma(k)) < \kappa$. Simultaneously, we will construct 
  sets $y_\sigma$ for $\sigma \in T$ satisfying:
  \begin{enumerate}
    \item for all $i \leq m+1$ and all $\sigma \in T_i$, we have $y_\sigma \in \mathcal{P}_{\kappa^{+m+1-i}}(\kappa^{+m})$ 
      and, moreover, if $i \leq m$, then $|y_\sigma| = \kappa^{+m-i}$;
    \item for all $i \leq m+1$ and all $\sigma \in T_i$, the following hold:
      \begin{itemize} 
        \item for all $z \in [y_\sigma]^{<\omega}$, we have $F(z) \subseteq y_\sigma$;  
        \item $y_\sigma \cap \kappa^{+m+1-i} \in \kappa^{+m+1-i}$;
        \item for all $k < i$, $\sup(y_\sigma \cap \kappa^{+m-k}) = \sigma(k)$;
      \end{itemize}
    \item for all $i \leq m$ and all $\sigma \in T_i$, we have that $\langle y_{\sigma ^\frown \langle \alpha \rangle} \mid 
      \alpha \in \mathrm{succ}_T(\sigma) \rangle$ is $\subseteq$-increasing and $(<\kappa)$-continuous, 
      and $\bigcup_{\alpha \in \mathrm{succ}_T(\sigma)} y_{\sigma ^\frown \langle \alpha \rangle} = y_\sigma$.
  \end{enumerate}

  To start, let $T_0 = \{\emptyset\}$ and $y_\emptyset = \kappa^{+m}$. Next, suppose $i \leq m$ and we have constructed 
  $T_i$ and $\{y_\sigma \mid \sigma \in T_i\}$ satisfying the recursion requirements listed above. We will construct 
  $T_{i+1}$ and $\{y_\sigma \mid \sigma \in T_{i+1}\}$.

  Fix $\sigma \in T_i$. $|y_\sigma| = \kappa^{+m-i}$, so we can find a $\subseteq$-increasing, continuous sequence 
  $\langle z_{\sigma, \alpha} \mid \alpha < \kappa^{+m-i} \rangle$ such that:
  \begin{itemize}
    \item for all $\alpha < \kappa^{+m-i}$, $z_{\sigma, \alpha} \in \mathcal{P}_{\kappa^{+m-i}}(y_\sigma)$ and, 
      moreover, if $i < m$, we additionally have $|z_{\sigma, \alpha}| = \kappa^{+m-i-1}$;
    \item $\bigcup_{\alpha < \kappa^{+m-i}} z_{\sigma, \alpha} = y_\sigma$;
    \item for all $k < i$, $\sup(z_{\sigma, 0} \cap \kappa^{+m-k}) = \sigma(k) = \sup(y_\sigma \cap \kappa^{+m-k})$;
    \item for all $\alpha < \kappa^{+m-i}$, $z_{\sigma, \alpha} \cap \kappa^{+m-i} \in \kappa^{+m-i}$;
    \item for all $\alpha < \kappa^{+m-i}$ and all $w \in [z_{\sigma, \alpha}]^{<\omega}$, we have 
      $F(w) \subseteq z_{\sigma, \alpha}$.
  \end{itemize}
  Let $\mathrm{succ}_T(\sigma) = \{\alpha \in S^{\kappa^{+m-i}}_{<\kappa} \mid z_{\sigma, \alpha} \cap 
  \kappa^{+m-i} = \alpha\}$ and, for $\alpha \in \mathrm{succ}_T(\sigma)$, let $y_{\sigma ^\frown \langle \alpha \rangle} 
  = z_{\sigma, \alpha}$. This completes the definition of $T_{i+1}$ and $\{y_\sigma \mid \sigma \in T_{i+1}\}$; it is 
  straightforward to verify that the recursion hypotheses have been maintained. It is now immediate that, 
  for all $\sigma \in T_{m+1}$, we have $y_\sigma \in \mathcal{C}_F \subseteq \mathcal{C}$ and, for all $i \leq m$, we have 
  $\sup(y_\sigma \cap \lambda_i) = \sigma(i)$.
\end{proof}

We are now ready to prove Theorem \ref{fat_tree_guessing_theorem}.  The main idea of the proof is that,
working in $V$, we exploit the regularity of $\kappa$ to produce structures that
give rise to PCF-theoretic objects in $W$.  The conclusion comes from a careful
analysis of these PCF-theoretic objects.

\begin{proof}[Proof of Theorem \ref{fat_tree_guessing_theorem}]
  Work first in $V$. Let $X = (\mathcal{P}_\kappa(\lambda))^V$. We first introduce some notation. 
  If $i \leq m, ~ f:X \rightarrow \lambda_i$, 
  and $D \subseteq \lambda_i$ is unbounded, let $f^D:X \rightarrow \lambda_i$ be defined by $f^D(x) = 
  \min(D \setminus f(x))$ for all $x \in X$. Fix a sequence $\langle e_\beta \mid \beta < \lambda^+ \rangle$ 
  such that, for all $\beta < \lambda^+$, $e_\beta:\beta \rightarrow \lambda$ is injective. For all $\alpha < \lambda^+$ 
  and all $\sigma \in T(\alpha)$ such that $\lh(\sigma) \leq m$, let $D_{\alpha, \sigma} = \mathrm{succ}_{T(\alpha)}(\sigma)$. 
  $D_{\alpha, \sigma}$ is thus $(<\kappa)$-club in $\lambda_{\lh(\sigma)}$.

  For all $i \leq m$, define a sequence of functions $\vec{f}_i = \langle f_{i, \beta} \mid \beta < \lambda^+ \rangle$ satisfying 
  the following requirements:
  \begin{enumerate}
    \item for all $\beta < \lambda^+, ~ f_{i, \beta}:X \rightarrow \lambda_i$;
    \item for all $\beta < \gamma < \lambda^+$ and all $x \in X$, if $e_\gamma(\beta) \in x$, then 
      $f_{i, \beta}(x) < f_{i, \gamma}(x)$;
    \item for all $\gamma \in S^{\lambda^+}_{<\kappa}$, there is a club $c_\gamma \subseteq \gamma$ such that
      $\sup\{f_{i, \beta} \mid \beta \in c_\gamma\} < f_{i, \gamma}$;
    \item for all $\alpha, \beta < \lambda^+$, there is $\gamma < \lambda^+$ such that, for all $x \in X$ and 
      all $\sigma \in T(\alpha)_i$ such that $\mathrm{range}(\sigma) \subseteq x$, we have $f^{D_{\alpha, \sigma}}_{i, \beta}(x) 
      < f_{i, \gamma}(x)$.
  \end{enumerate}
  The construction is straightforward, by recursion on $\beta < \lambda^+$.

  Move now to $W$. Let $\theta = |Y|$, and let $\mu = (\lambda^+)^V$. Note that, by condition (4) in the statement 
  of the theorem, $\mu = |\kappa|^+$ and, for every $V$-regular cardinal $\epsilon \in [\kappa, \lambda]$, 
  $\cf^W(\epsilon) \leq \theta$.

  We will define $\langle \delta_{i,y} \mid i \leq m, ~ y \in Y \rangle$ by recursion on $i$. Thus, suppose 
  $i \leq m$ and we have defined $\langle \delta_{j, y} \mid j < i, ~ y \in Y \rangle$. For $y \in Y$, 
  let $y_i = y \cup \{\delta_{j,y} \mid j < i\}$, and let $Y_i = \{y_i \mid y \in Y\}$. The following are 
  easily verified:
  \begin{itemize}
    \item $Y_i$ is a $\subseteq$-directed subset of $X$ and $\bigcup Y_i = \lambda$;
    \item if $I$ and $I_i$ are the non-$\subseteq$-cofinal ideals on $Y$ and $Y_i$, respectively, then, for all 
      $Z \subseteq Y$, $Z \in I$ iff $\{y_i \mid y \in Z\} \in I_i$.
  \end{itemize}
  Define a sequence of functions $\vec{g}_i = \langle g_{i, \beta} \mid \beta < \mu \rangle$ from $Y$ to $\lambda_i$ 
  by letting $g_{i, \beta}(y) = f_{i, \beta}(y_i)$ for all $\beta < \mu$ and $y \in Y$. By the requirements placed on 
  $\vec{f}_i$, we have the following:
  \begin{itemize}
    \item $\vec{g}_i$ is $<_I$-increasing;
    \item for all $\gamma < \mu$ such that $\cf^W(\beta) > \theta$, there is a club $c_\gamma \subseteq \gamma$ such 
      that $\sup\{g_{i, \beta} \mid \beta \in c_\gamma\} < g_{i, \gamma}$.
  \end{itemize}
  Therefore, by Lemma \ref{club_guessing_lemma}, $\vec{g}_i$ has an eub, $h_i$, such that, 
  for all $W$-regular $\nu$ such that $\theta < \nu < \nu^{++} < \mu$, 
  there is $y_{i, \nu} \in Y$ such that, for all $y \in Y / y_{i,\nu}$, we have $\cf(h_i(y)) > \nu$.
  We may assume that, for all $y \in Y$, $\cf(h_i(y)) > \theta$, which implies that $\cf^V(h_i(y)) < \kappa$ 
  and, in turn, $h_i(y) < \lambda_i$. For all $y \in Y$, let $\delta_{i, y} = h_i(y)$.

  We claim that $\langle \delta_{i,y} \mid i \leq m, ~ y \in Y \rangle$ is as
  desired. Requirement (ii) is immediate, 
  so it remains to verify (i). Suppose $\alpha < \mu$ is a counterexample to (i), and let $k \leq m$ be least 
  such that there is no $y \in Y$ such that, for all $z \in Y / y$, $\langle \delta_{i,z} \mid i \leq k \rangle \in T_\alpha$.
  For $z \in Y$, let $\sigma_z = \langle \delta_{i,z} \mid i < k \rangle$. By the minimality of $k$, 
  we can fix $y^* \in Y$ such that, for all $z \in Y / y^*$, $\sigma_z \in T_\alpha$. 
  Since there is no $y \in Y$ such that $\langle \delta_{i,z} \mid i \leq k \rangle \in T_\alpha$ for all $z \in Y / y$, 
  we can fix a set $A \in I^+$ (i.e., a $\subseteq$-cofinal set $A \subseteq Y$) such that, for all $z \in A, ~ 
  \langle \delta_{i,z} \mid i \leq k \rangle \not\in T_\alpha$.
  By shrinking $A$ if necessary, we may assume that $A \subseteq Y / y^*$, which implies that, for all $z \in A, ~ 
  \delta_{k,z} \not\in D_{\alpha, \sigma_z}$.

  Since, for all $z \in A$, $\cf^V(\delta_{k,z}) < \kappa$ and $D_{\alpha, \sigma_z}$ is $(<\kappa)$-club in $\lambda_k$, 
  we have $\sup(D_{\alpha, \sigma_z} \cap \delta_{k,z}) < \delta_{k,z}$. Define a function $\bar{h}: Y \rightarrow \lambda_k$ 
  by letting, for each $y \in Y$:
  \[
    \bar{h}(y) = \begin{cases} \sup(D_{\alpha, \sigma_z} \cap \delta_{k,z}) & \mbox{ if } z \in A \\
      0 & \mbox{ if } z \not\in A \end{cases}
  \]
  $\bar{h} < h_k$, so, since $h_k$ is an eub for $\vec{g}_k$, there is $\beta <
\mu$ such that $\bar{h} <_I g_{k,\beta}$.
  Let $\gamma < \mu$ be as given in requirement (4) applied to $\alpha$ and $\beta$ in the construction of 
  $\vec{f}_k$. Note that, for all $z \in A$, $\mathrm{range}(\sigma_z) \subseteq
z_k= z \cup \{ \delta_{j,z} \mid j <k\}$. Therefore, for all $z \in A$, 
  $\min(D_{\alpha, \sigma_z} \setminus g_{k,\beta}(z)) < g_{k,\gamma}(z)$. 

  Fix $y' \supseteq y^*$ such that, for all $z \in Y / y', ~ \bar{h}(z) < g_{k,\beta}(z)$. Then, for all 
  $z \in A / y'$, we have $\min(D_{\alpha, \sigma_z} \setminus (\bar{h}(z) + 1)) \leq \min(D_{\alpha, \sigma_z} \setminus g_{k,\beta}(z)) < g_{k,\gamma}(z)$. 
  However, by our definition of $\bar{h}$, we have that, for all $z \in A$, $\delta_{k,z} < \min(D_{\alpha, \sigma_z} \setminus (\bar{h}(z) + 1))$. 
  Therefore, for all $z \in A / y'$, we have $h_k(z) < g_{k,\gamma}(z)$, contradicting the fact that 
  $h_k$ is an eub for $\vec{g}_k$.
\end{proof}

We mention now some specific consequences of Theorem \ref{fat_tree_guessing_theorem}. 
The first is a generalization of Theorem \ref{dzamonja_shelah_theorem}.

\begin{corollary} \label{ms_cor}
  Suppose that:
  \begin{enumerate}
    \item $V$ is an inner model of $W$;
    \item in $V$, $\theta < \kappa \leq \lambda$ are regular cardinals;
    \item in $W$, $(\lambda^+)^V$ remains a cardinal and, for all $V$-regular cardinals 
      $\nu \in [\kappa, \lambda]$, $\cf(\nu) = \theta < |\kappa| = |\lambda|$;
    \item in $V$, $\langle C_\alpha \mid \alpha < \lambda^+ \rangle$ is a sequence of clubs in $\kappa$;
    \item one of the following holds:
      \begin{enumerate}
        \item $\lambda < \kappa^{+\omega}$;
        \item $\theta = \omega$ and $\lambda < \kappa^{+\omega_1}$.
      \end{enumerate}
  \end{enumerate}
  Then, in $W$, there is a sequence of ordinals $\langle \delta_i \mid i < \theta \rangle$ such that:
  \begin{enumerate}
    \item[(i)] for all $\alpha < (\lambda^+)^V$, for all sufficiently large $i < \theta$, we have $\delta_i \in C_\alpha$;
    \item[(ii)] for all $\gamma < |\kappa|$, for all sufficiently large $i < \theta$, we have $\cf(\delta_i) > \gamma$.
  \end{enumerate}
\end{corollary}

\begin{proof}
  First note that, by Lemma \ref{shelah_lemma}, $|\kappa|$ is a singular cardinal of cofinality $\theta$ in $W$. By Lemmas 
  \ref{first_covering_lemma} and \ref{second_covering_lemma}, either case (5)(a) or (5)(b) of the hypothesis implies that, 
  in $W$, there is a $\subseteq$-increasing 
  sequence $\langle y_i \mid i < \theta \rangle$ from $(\mathcal{P}_\kappa(\lambda))^V$ such that $\bigcup_{i < \theta} y_i = \lambda$.
  The conclusion now follows immediately from Theorem \ref{fat_tree_guessing_theorem}.
\end{proof}

The next result is a generalization of Theorem \ref{magidor_sinapova_theorem}.

\begin{corollary} \label{magidor_sinapova_cor}
  Suppose that:
  \begin{enumerate}
    \item $V$ is an inner model of $W$;
    \item in $V$, $\theta < \kappa$ are regular cardinals;
    \item $m < \omega$, and $\kappa^{+m+1}$ remains a cardinal in $W$;
    \item for all $i \leq m$, $\cf^W((\kappa^{+i})^V) = \theta < |\kappa|$;
    \item in $V$, $\langle \mathcal{C}_\alpha \mid \alpha < \kappa^{+m+1} \rangle$ is a sequence 
      of clubs in $\mathcal{P}_\kappa(\kappa^{+m})$.
  \end{enumerate}
  Then, in $W$, there is a sequence $\langle x_\eta \mid \eta < \theta \rangle$ such that:
  \begin{enumerate}
    \item[(i)] for all $\alpha < (\kappa^{+m+1})^V$ and all sufficiently large $\eta < \theta$, we have $x_\eta \in \mathcal{C}_\alpha$;
    \item[(ii)] for all $\gamma < |\kappa|$, for all sufficiently large $\eta < \theta$, for all $i \leq m$, we have  
      $\cf^W(\sup(x_\eta \cap (\kappa^{+i})^V)) > \gamma$.
  \end{enumerate}
\end{corollary}

\begin{proof}
  Work first in $V$. Let $\Upsilon$ be a sufficiently large regular cardinal, and, as in Section \ref{prelim_section}, 
  let 
  \[
  \mathcal{X} := \{M \cap \kappa^{+m} \mid M \prec H(\Upsilon), ~  
  |M| < \kappa,\mbox{ and }M \cap \kappa \in \kappa\}.
  \] 
  By shrinking the clubs if necessary, we may assume that, for all $\alpha < \kappa^{+m+1}$, we have  
  $\mathcal{C}_\alpha \subseteq \mathcal{X}$. For all $\alpha < \kappa^{+m+1}$, let $T(\alpha)$ be the fat tree of type 
  $(\kappa, \langle \kappa^{+m}, \kappa^{+m-1}, \ldots, \kappa \rangle)$ given by Lemma \ref{fat_tree_model_lemma} 
  applied to $\mathcal{C}_\alpha$. 

  By Lemma \ref{second_covering_lemma}, there is, in $W$, a $\subseteq$-increasing sequence $\langle y_\eta \mid \eta < \theta \rangle$ 
  of elements from $(\mathcal{P}_\kappa((\kappa^{+m})^V))^V$ such that $\bigcup_{\eta < \theta} y_\eta = (\kappa^{+m})^V$. 
  Also, by Lemma \ref{shelah_lemma}, $|\kappa|$ is a singular cardinal of cofinality $\theta$ in $W$.
  Therefore, we may apply Theorem \ref{fat_tree_guessing_theorem} to find $\langle \delta_{i, \eta} \mid i \leq m, ~ \eta < \theta \rangle$ 
  such that:
  \begin{itemize}
    \item for all $\alpha < (\kappa^{+m+1})^V$, for all sufficiently large $\eta < \theta, ~ \langle \delta_{i,\eta} \mid i \leq m \rangle 
      \in T(\alpha)$;
    \item for all $\gamma < |\kappa|$, for all sufficiently large $\eta < \theta$, for all $i \leq m, ~ \cf^W(\delta_{i, \eta}) > \gamma$.
  \end{itemize}
  Without loss of generality, we may assume that, for all $\eta < \theta$, there is $x \in \mathcal{X}$ such that, for all $i \leq m$, 
  $\chi_x(m-i) = \delta_{i, \eta}$ and that, for all $\eta < \theta$ and $i \leq m$, $\cf^W(\delta_{i, \eta}) > \omega$. Therefore, 
  by Lemma \ref{uniqueness_lemma}, we in fact have that, for all $\eta < \theta$, there is a \emph{unique} $x \in \mathcal{X}$ such that, 
  for all $i \leq m$, $\chi_x(m-i) = \delta_{i, \eta}$. Let $x_\eta$ be this unique $x$.

  We claim that $\langle x_\eta \mid \eta < \theta \rangle$ is as desired. Requirement (ii) follows immediately from the properties of 
  $\langle \delta_{i, \eta} \mid i \leq m, ~ \eta < \theta \rangle$. To see (i), fix $\alpha < (\kappa^{+m+1})^V$. For all sufficiently large 
  $\eta < \theta$, by Lemma \ref{fat_tree_model_lemma}, the definition of $T(\alpha)$, and the properties of $\langle \delta_{i, \eta} 
  \mid i \leq m, ~ \eta < \theta \rangle$, there is $z \in \mathcal{C}_\alpha$ such that, for all $i \leq m$, $\chi_z(i) = \delta_{i, \eta}$. 
  By the fact that $\mathcal{C}_\alpha \subseteq \mathcal{X}$ and the preceding discussion, there is in fact a unique such $z$, and it is 
  $x_\eta$. Therefore, for all sufficiently large $\eta < \theta$, $x_\eta \in \mathcal{C}_\alpha$.
\end{proof}

\section{Diagonal sequences} \label{diagonal_section}

In this section, we prove abstract results about the existence of diagonal pseudo-Prikry sequences that 
are natural generalizations of the objects added by diagonal Prikry-type forcings, such as the diagonal 
supercompact Prikry forcing from \cite{gitik_sharon}. In order to formulate such results, we need a replacement for the 
notion of ``club'' that is applicable in the context of singular cardinals.

\begin{definition}
  Suppose $\mu$ is a singular cardinal, $\cf(\mu) = \theta$, and $\vec{\mu} = \langle \mu_\xi \mid \xi < \theta \rangle$ 
  is an increasing sequence of regular cardinals, cofinal in $\mu$. We say $\vec{C} = 
  \langle C_\xi \mid \xi < \theta \rangle$ is a \emph{diagonal club in $\vec{\mu}$} if, for every $\xi < \theta$, 
  $C_\xi$ is club in $\mu_\xi$.
\end{definition}

The following lemma is a consequence of work of Cummings
\cite{cummings_collapsing} and a remark of Sharon and Viale \cite{sharon_viale}.
We give a self-contained proof.

\begin{lemma} \label{csv_lemma}
  Suppose that:
  \begin{enumerate}
    \item $V$ is an inner model of $W$;
    \item in $V$, $\mu$ is a singular cardinal and $\theta = \cf(\mu)$;
    \item $(\mu^+)^V$ remains a cardinal in $W$;
    \item $\kappa = \cf^W(\theta)$.
  \end{enumerate}
  Then, in $W$, there is a cardinal $\nu$ and a $k \leq 2$ such that $\cf(\nu) = \kappa$ 
  and $(\mu^+)^V = \nu^{+k+1}$.
\end{lemma}

\begin{proof}
  Work first in $V$. Let $\lambda = \mu^+$. Apply Theorem \ref{scale_thm} to find an increasing 
  sequence of regular cardinals, $\vec{\mu} = \langle \mu_\xi \mid \xi < \theta \rangle$, 
  such that $\vec{\mu}$ is cofinal in $\mu$ and there is a scale in $\prod_{\xi < \theta} \mu_\xi$ of length $\lambda$. 
  Let $\vec{f} = \langle f_\alpha \mid \alpha < \lambda \rangle$ be such a scale. 
  By making adjustments to the scale if necessary, we may assume that, for every limit ordinal $\beta < \lambda$, there 
  is a club $c_\beta \subseteq \beta$ such that $\sup\{f_\alpha \mid \alpha \in c_\beta\} <^* f_\beta$.\footnote{See, 
    for example, the proof of \cite[Theorem 2.21]{abraham_magidor} for details on how to achieve this.}
  
  Move now to $W$. Let $\langle \xi_i \mid i < \kappa \rangle$ be an increasing sequence of ordinals, cofinal in $\theta$, 
  and define a sequence of functions $\vec{g} = \langle g_\alpha \mid \alpha < \lambda \rangle$ from $\kappa$ to $\mu$ 
  by letting $g_\alpha(i) = f_\alpha(\xi_i)$ for all $\alpha < \lambda$ and $i < \kappa$. 

  Suppose first that there is a cardinal $\nu \geq \kappa$ such that $\lambda = \nu^{+4}$. Then, by Lemmas \ref{small_cf_lemma} and \ref{club_guessing_lemma},  
  $\vec{g}$ has an eub, $h$, such that, for all $i < \kappa$, $\cf(h(i)) > \nu^{+3}$. 
  But the function $i \mapsto \mu_{\xi_i}$ is an upper bound for $\vec{g}$, so, for all sufficiently large $i < \kappa$, 
  $h(i) \leq \mu_{\xi_i} < \mu$. This is a contradiction, as $|\mu| = \nu^{+3}$.

  Therefore, as $\lambda > \kappa$ and $\lambda$ is a successor cardinal, there must be a cardinal $\nu$ and a $k \leq 2$ 
  such that $\lambda = \nu^{+k+1}$ and either $\nu = \kappa$ or $\nu$ is a limit cardinal. If $\nu = \kappa$, then we are done, 
  so suppose $\nu$ is a limit cardinal and, for sake of contradiction, $\cf(\nu) \neq \kappa$.
  By Lemma \ref{club_guessing_lemma}, $\vec{g}$ has an eub, $h$, such that, for all $\epsilon < \nu$, for all sufficiently 
  large $i < \kappa$, $\cf(h(i)) > \epsilon$.
  
  Suppose first that there is an unbounded set $A \subseteq \kappa$ such that, for all 
  $i < \kappa$, $\cf(h(i)) < \nu$. Since $\kappa \neq \cf(\nu)$, we may find an unbounded $B \subseteq A$ and an $\epsilon < \nu$ 
  such that, for all $i \in B$, $\cf(h(i)) < \epsilon$, contradicting the fact that, for all sufficiently large $i < \kappa$, 
  $\cf(h(i)) > \epsilon$. 

  Thus, we may assume that, for all sufficiently large $i < \kappa$, $\cf(h(i)) \geq \nu$. Then, by Lemma \ref{small_cf_lemma}, 
  for all sufficiently large $i < \kappa$, $\cf(h(i)) = \nu^{+k+1}$. This leads to a contradiction as before, as we must have 
  $h(i) < \mu$ for all sufficiently large $i < \kappa$ and $|\mu| = \nu^{+k}$.
\end{proof}

We can now prove a diagonal version of Theorem \ref{dzamonja_shelah_theorem}. Note that, in contrast to the situation in that theorem, 
we obtain here a pseudo-Prikry sequence that simultaneously meets \emph{every} diagonal club in $V$, not just every member of a 
predetermined list of $\mu^+$-many diagonal clubs.

\begin{theorem} \label{diag_thm1}
  Suppose that:
  \begin{enumerate}
    \item $V$ is an inner model of $W$;
    \item in $V$, $\mu$ is a singular cardinal and $\theta = \cf(\mu)$;
    \item in $V$, $\vec{\mu} = \langle \mu_\xi \mid \xi < \theta \rangle$ is an increasing sequence 
      of regular cardinals, cofinal in $\mu$, such that there is a scale of length $\mu^+$ in $\prod_{\xi < \theta} \mu_\xi$;
    \item in $W$, $(\mu^+)^V = (\nu)^{+k+1}$ for some singular cardinal $\nu$ and $k \leq 2$;       
    \item in $W$, $\sup(\{\cf(\mu_\xi) \mid \xi < \theta\} \cap \nu\}) < \nu$.
  \end{enumerate}
  Then, in $W$, there is a function $g \in \prod_{\xi < \theta} \mu_\xi$ such that:
  \begin{enumerate}
    \item[(i)] for all $\vec{C} = \langle C_\xi \mid \xi < \theta \rangle \in V$ such that $\vec{C}$ is a diagonal club in $\vec{\mu}$, 
      for all sufficiently large $\xi < \theta$, we have $g(\xi) \in C_\xi$;
    \item[(ii)] for all $\epsilon < \nu$, for all sufficiently large $\xi < \theta$, we have $\epsilon < \cf(g(\xi)) < \nu$.
  \end{enumerate}
\end{theorem}

\begin{proof}
  Let $\lambda = (\mu^+)^V$. In $V$, fix a scale $\vec{f} = \langle f_\alpha \mid \alpha < \lambda \rangle$ in 
  $\prod_{\xi < \theta} \mu_\xi$. As in the proof of Lemma \ref{csv_lemma}, we may assume that, for every limit ordinal $\beta < \lambda$, there is a club 
  $c_\beta \subseteq \beta$ such that $\sup\{f_\alpha \mid \alpha \in c_\beta\} <^* f_\beta$.
  
  Move to $W$. By now-familiar arguments, $\vec{f}$ has an eub $g$ such that, for all $\epsilon < \nu$, for all sufficiently large 
  $\xi < \theta$, we have $\epsilon < \cf(g(\xi))$. Suppose first that $A:=\{\xi < \theta \mid \cf(g(\xi)) \geq \nu\}$ is unbounded 
  in $\theta$. Then, by restricting all functions in $\vec{f}$ to $A$, Lemma \ref{small_cf_lemma} implies that, for 
  sufficiently large $\xi \in A$, $\cf(g(\xi)) = \lambda$. This is a contradiction, since the function $\xi \mapsto \mu_\xi$ 
  is an upper bound for $\vec{f}$ and, for all $\xi < \lambda$, we have $\mu_\xi < \lambda$. Therefore, we may assume that, 
  for all $\xi < \theta$, we have $\cf(g(\xi)) < \nu$. Since $\sup(\{\cf(\mu_\xi) \mid \xi < \theta\} \cap \nu\}) < \nu$, we may 
  assume further that, for all $\xi < \theta$, $g(\xi) < \mu_\xi$. We claim that $g$ is as desired. We have already shown 
  that $g$ satisfies requirement (ii), so it remains to verify (i). 

  To this end, let $\vec{C} = \langle C_\xi \mid \xi < \theta \rangle \in V$ be a diagonal club in $\vec{\mu}$, 
  and suppose for sake of contradiction that $B := \{\xi < \theta \mid g(\xi) \not\in C_\xi\}$ is unbounded in $\theta$. 
  Define a function $h \in \prod_{\xi < \theta} \mu_\xi$ by:
  \[
    h(\xi) = \begin{cases} \max(C_\xi \cap g(\xi)) & \mbox{ if } \xi \in B \\ 
      0 & \mbox{ if } \xi \not\in B. \end{cases}
  \]
  $h < g$, so, since $g$ is an eub, there is $\alpha < \lambda$ such that $h <^* f_\alpha$. Define a function 
  $\hat{h} \in \prod_{\xi < \theta} \mu_\xi$ by letting $\hat{h}(\xi) = \min(C_\xi \setminus f_\alpha(\xi))$ for all $\xi < \theta$. Since 
  $\vec{C}$ and $f_\alpha$ are in $V$, we also have $\hat{h} \in V$, so, as $\vec{f}$ is a scale in $V$, there is $\beta < \lambda$ such 
  that $\hat{h} <^* f_\beta$. Fix $\xi^* < \theta$ such that, for all $\xi \in \theta \setminus \xi^*$, we have  
  $h(\xi) < f_\alpha(\xi) \leq \hat{h}(\xi) < f_\beta(\xi)$. Then, for all $\xi \in B \setminus \xi^*$, 
  \[
    g(\xi) < \min(C_\xi \setminus g(\xi)) = \min(C_\xi \setminus h(\xi) + 1) \leq \min(C_\xi \setminus f_\alpha(\xi)) = \bar{h}(\xi) 
    < f_\beta(\xi),
  \] 
  contradicting the fact that $f_\beta <^* g$.
\end{proof}

Models $V$ and $W$ as in the statement of Theorem \ref{diag_thm1} can be obtained, for $\theta = \aleph_0$, using the diagonal supercompact Prikry forcing 
introduced by Gitik and Sharon in \cite{gitik_sharon} and, for uncountable $\theta$, using the diagonal supercompact Magidor forcing 
introduced by Sinapova in \cite{sinapova_2008}. 

We also obtain the following variant of Theorem \ref{diag_thm1}. We thank the referee for bringing it to our attention.

\begin{theorem} \label{diag_thm1.5}
  Suppose that:
  \begin{enumerate}
    \item $V$ is an inner model of $W$;
    \item in both $V$ and $W$, $\mu$ is a singular cardinal and $\theta = \mathrm{cf}(\mu)$;
    \item in $V$, $\vec{\mu} = \langle \mu_\xi \mid \xi < \theta \rangle$ is an increasing sequence 
      of regular cardinals, cofinal in $\mu$, such that there is a scale of length $\mu^+$ in $\prod_{\xi < \theta} \mu_\xi$;
    \item $(\mu^+)^V = (\mu^+)^W$;
    \item $(\prod_{\xi < \theta} \mu_\xi)^V$ is bounded in $((\prod_{\xi < \theta} \mu_\xi)^W, <^*)$.
  \end{enumerate}
  Then, in $W$, there is a function $g \in \prod_{\xi < \theta} \mu_\xi$ such that:
  \begin{enumerate}
    \item[(i)] for all $\vec{C} = \langle C_\xi \mid \xi < \theta \rangle \in V$ such that $\vec{C}$ is a diagonal club in $\vec{\mu}$, 
      for all sufficiently large $\xi < \theta$, $g(\xi) \in C_\xi$;
    \item[(ii)] for all $\epsilon < \mu$, for all sufficiently large $\xi < \theta$, $\epsilon < \cf(g(\xi))$.
  \end{enumerate}
\end{theorem}

The proof of Theorem \ref{diag_thm1.5} is essentially the same as that of Theorem \ref{diag_thm1}, so we omit it. We note that, 
for $\theta = \aleph_0$, models $V$ and $W$ as in its statement can be obtained using diagonal Prikry forcing 
or extender-based forcings such as those in \cite[Sections 1.3 and 2]{gitikhandbook}.

We can now extend Corollary \ref{magidor_sinapova_cor} to the case in which $\kappa^{+\omega+1}$ becomes the 
successor of $|\kappa|$. We first note that, by the following result of Gitik, a straightforward generalization to 
clubs in $\mathcal{P}_\kappa(\kappa^{+\omega})$ is impossible. We therefore seek a diagonal sequence 
meeting diagonal clubs in $\langle \mathcal{P}_\kappa(\kappa^{+i}) \mid i < \omega \rangle$.

\begin{proposition}[Gitik, \cite{gitik_witnessing_seq}, Proposition 0.4]
  Suppose that:
  \begin{enumerate}
    \item $V$ is an inner model of $W$;
    \item $\kappa < \mu$ are cardinals in $V$, with $\kappa$ regular and $\cf(\mu) < \kappa$;
    \item in $V$, for all $\tau < \kappa$, $\tau^{\cf(\mu)} \leq \mu$;
    \item $(\mu^+)^V$ remains a cardinal in $W$.
  \end{enumerate}
  Then there is a sequence $\langle \mathcal{C}_\alpha \mid \alpha < (\mu^+)^V \rangle \in V$ of clubs in 
  $(\mathcal{P}_\kappa(\mu))^V$ such that, in $W$, for any $\theta < (\mu^+)^V$ and any sequence 
  $\langle x_i \mid i < \theta \rangle$ of elements of $(\mathcal{P}_\kappa(\mu))^V$, there is 
  $\alpha < (\mu^+)^V$ such that, for all $i < \theta$, $x_i \not\in \mathcal{C}_\alpha$.
\end{proposition}

\begin{theorem} \label{diagonal_guessing_theorem}
  Suppose that:
  \begin{enumerate}
    \item $V$ is an inner model of $W$;
    \item $\kappa$ is a regular cardinal in $V$;
    \item for all $i < \omega$, $\cf^W((\kappa^{+i})^V) = \omega$;
    \item in $W$, $(\kappa^{+\omega+1})^V = \nu^{+k+1}$ for some singular cardinal $\nu$ and $k \leq 2$;
    \item $(2^{\aleph_0})^W < (\kappa^{+\omega+1})^V$;
    \item $\langle \mathcal{C}_{\alpha, i} \mid \alpha < (\kappa^{+\omega+1})^V, ~ i < \omega \rangle \in W$ 
      is such that:
      \begin{enumerate}
        \item for all $i < \omega$, $\langle \mathcal{C}_{\alpha, i} \mid 
          \alpha < (\kappa^{+\omega+1})^V \rangle \in V$;
        \item for all $\alpha < (\kappa^{+\omega+1})^V$ and $i < \omega, ~ \mathcal{C}_{\alpha, i}$ is club in 
          $(\mathcal{P}_\kappa(\kappa^{+i}))^V$.
      \end{enumerate}
  \end{enumerate}
  Then, in $W$, there is a sequence $\langle x_i \mid i < \omega \rangle$ such that:
  \begin{enumerate}
    \item[(i)] for all $\alpha < (\kappa^{+\omega+1})^V$ and all sufficiently large $i < \omega, ~ x_i \in \mathcal{C}_{\alpha, i}$.
    \item[(ii)] for all $\gamma < \nu$, all sufficiently large $i < \omega$, and all $j \leq i$, $\cf(\sup(x_i \cap (\kappa^{+j})^V)) > \gamma$.
  \end{enumerate}
\end{theorem}

\begin{proof}
  In $V$, let $\Upsilon$ be a sufficiently large, regular cardinal, let $\lambda = \kappa^{+\omega+1}$, 
  let $\mu = \kappa^{+\omega}$, and, for $i < \omega$, let $\mu_i = \kappa^{+i}$. 
  By Lemma \ref{first_covering_lemma}, there is, in $W$, a $\subseteq$-increasing sequence 
  $\langle y_i \mid i < \omega \rangle$ from $(\mathcal{P}_\kappa(\mu))^V$ such 
  that $\bigcup_{i < \omega} y_i = \mu$. In $V$, for $i < \omega$, let $\mathcal{D}_i$ be the 
  set of $x \in \mathcal{P}_\kappa(\mu_i)^V$ such that:
  \begin{itemize}
    \item there is $M \prec H(\Upsilon)$ such that $x = M \cap \mu_i$;
    \item $x \cap \kappa \in \kappa$;
    \item $(y_i \cap \mu_i) \subseteq x$.
  \end{itemize}
  $\mathcal{D}_i$ is club in $\mathcal{P}_\kappa(\mu_i)$ and $\mathcal{D}_i \in V$ (though $\langle \mathcal{D}_i \mid i < \omega \rangle \not\in V$). 
  By intersecting each $\mathcal{C}_{\alpha, i}$ with $\mathcal{D}_i$, we may assume that, for all $\alpha < \lambda$ 
  and all $i < \omega$, $\mathcal{C}_{\alpha, i} \subseteq \mathcal{D}_i$. Note that this does not interfere with the hypotheses 
  of the theorem, as it is still the case that, for all $i < \omega$, $\langle \mathcal{C}_{\alpha, i} \mid 
  \alpha < \lambda \rangle \in V$.

  Still working in $V$, fix a sequence $\langle e_\beta \mid \beta < \lambda \rangle$ such that, for all 
  $\beta < \lambda, ~ e_\beta:\beta \rightarrow \mu$ is an injection. For each $i < \omega$, by recursion on $\alpha < \lambda$, 
  define $\langle \mathcal{C}'_{\alpha, i} \mid \alpha < \lambda \rangle$ as follows. If $\beta < \lambda$ and 
  $\langle \mathcal{C}'_{\alpha, i} \mid \alpha < \beta \rangle$ has been defined, let 
  \[
    \mathcal{C}'_{\beta, i} = \mathcal{C}_{\beta, i} \cap \{x \in \mathcal{P}_\kappa(\mu_i) \mid \mbox{for 
    all } \alpha < \beta, \mbox{ if } e_\beta(\alpha) \in x, \mbox{ then } x \in \mathcal{C}'_{\alpha, i}\}.
  \]
  The following facts are immediate:
  \begin{itemize}
    \item for all $i < \omega$, $\langle \mathcal{C}'_{\alpha, i} \mid \alpha < \lambda \rangle \in V$;
    \item for all $\alpha < \lambda$ and $i < \omega$, $\mathcal{C}'_{\alpha, i} \subseteq \mathcal{C}_{\alpha, i}$ and 
      $\mathcal{C}'_{\alpha, i}$ is club in $\mathcal{P}_\kappa(\mu_i)$;
    \item for all $\alpha < \beta < \lambda$ and all $i < \omega$ such that $e_\beta(\alpha) \in (y_i \cap \mu_i)$, we have  
      $\mathcal{C}'_{\beta, i} \subseteq \mathcal{C}'_{\alpha, i}$.
  \end{itemize}
  
  Move now to $W$. By Theorem \ref{fat_tree_guessing_theorem} and the arguments from the proof of Corollary \ref{magidor_sinapova_cor}, 
  there is, for each $i < \omega$, a $\subseteq$-increasing sequence $\langle z_{i,n} \mid n < \omega \rangle$ such that:
  \begin{itemize}
    \item for all $\alpha < \lambda$ and all sufficiently large $n < \omega$, $z_{i,n} \in \mathcal{C}'_{\alpha, i}$;
    \item for all $\gamma < \nu$, all sufficiently large $n < \omega$, and all $j \leq i$, $\cf(\sup(z_{i,n} \cap \mu_j)) > \gamma$.
  \end{itemize}
  Let $\langle \gamma_i \mid i < \omega \rangle$ be increasing and cofinal in $\nu$. For all $\alpha < \lambda$, 
  find a function $\sigma_\alpha \in {^\omega}\omega$ such that, for all $i < \omega$, all $n \geq \sigma_\alpha(i)$, and all $j \leq i$, we have 
  $z_{i, n} \in \mathcal{C}'_{\alpha, i}$ and $\cf(\sup(z_{i,
  n} \cap \mu_j)) > \gamma_i$. 
  Since $2^{\aleph_0} < \lambda$, we can find an unbounded $A \subseteq \lambda$ and a fixed $\sigma \in {^\omega}\omega$ such 
  that, for all $\alpha \in A$, $\sigma_\alpha = \sigma$. For $i < \omega$, let $x_i = z_{i, \sigma(i)}$. 

  We claim that $\langle x_i \mid i < \omega \rangle$ is as desired. Requirement (ii) is immediate, since we have arranged that, 
  for all $i < \omega$ and all $j \leq i$, $\cf(\sup(x_i \cap \mu_j)) > \gamma_i$. To verify requirement (i), fix an 
  $\alpha < \lambda$. If $\alpha \in A$, then we have arranged that, for all $i < \omega$, $x_i \in \mathcal{C}'_{\alpha, i} 
  \subseteq \mathcal{C}_{\alpha, i}$. If $\alpha \not\in A$, then let $\beta = \min(A \setminus \alpha)$, and 
  let $i^* < \omega$ be least such that $e_\beta(\alpha) \in y_{i^*} \cap \mu_{i^*}$. Then, for all $i^* \leq i < \omega$, 
  we have $x_i \in \mathcal{C}'_{\beta, i} \subseteq \mathcal{C}'_{\alpha, i} \subseteq \mathcal{C}_{\alpha, i}$.
\end{proof}

In \cite{gitik_witnessing_seq}, Gitik extends Corollary \ref{magidor_sinapova_cor} under some additional 
cardinal arithmetic assumptions as follows.  The main difference here is that
$\kappa^{+m+1}$ has been replaced with a cardinal $\mu$ such that $\mu^{<\mu}=\mu$.

\begin{theorem}[Gitik, \cite{gitik_witnessing_seq}, Theorem 0.2]
  Suppose that:
  \begin{enumerate}
    \item $V$ is an inner model of $W$;
    \item in $V$, $\kappa < \mu$ are regular uncountable cardinals and $\mu^{<\mu} = \mu$;
    \item in $W$, there is a sequence $\langle y_i \mid i < \omega \rangle$ from $(\mathcal{P}_\kappa(\mu))^V$ 
      such that $\bigcup_{i < \omega} y_i = \mu$;
    \item $(\mu^+)^V$ remains a cardinal in $W$;
    \item $\mu \geq (2^{\aleph_0})^W$;
    \item in $V$, $\langle \mathcal{C}_\alpha \mid \alpha < \mu^+ \rangle$ is a sequence of clubs in $\mathcal{P}_\kappa(\mu)$.
  \end{enumerate}
  Then, in $W$, there is a sequence $\langle x_i \mid i < \omega \rangle$ such that, for all $\alpha < (\mu^+)^V$ and all 
  sufficiently large $i < \omega$, $x_i \in \mathcal{C}_\alpha$.
\end{theorem}

We can similarly extend Theorem \ref{diagonal_guessing_theorem} by replacing
$\kappa^{+\omega}$ with a singular, strong limit cardinal $\mu$.

\begin{theorem} \label{theorem_58}
  Suppose that:
  \begin{enumerate}
    \item $V$ is an inner model of $W$;
    \item in $V$, $\kappa$ is a regular cardinal;
    \item in $V$, $\mu > \kappa$ is a singular, strong limit cardinal with $\theta := \cf(\mu) < \kappa$;
    \item in $V$, $\langle \mu_i \mid i < \theta \rangle$ is an increasing sequence of regular cardinals, cofinal 
      in $\mu$, such that $\kappa \leq \mu_0$;
    \item in $W$, there is a $\subseteq$-increasing sequence $\langle y_i \mid i < \theta \rangle$ from $(\mathcal{P}_\kappa(\mu))^V$ 
      such that $\bigcup_{i < \theta} y_i = \mu$;
    \item in $W$, $(\mu^+)^V$ remains a cardinal and $\mu \geq 2^\theta$;
    \item $\langle \mathcal{C}_{\alpha, i} \mid \alpha < (\mu^+)^V, ~ i < \theta \rangle \in W$ is such that:
      \begin{enumerate}
        \item for all $i < \theta$, $\langle \mathcal{C}_{\alpha, i} \mid \alpha < (\mu^+)^V \rangle \in V$;
        \item for all $\alpha < (\mu^+)^V$ and all $i < \theta$, $\mathcal{C}_{\alpha, i}$ is club in 
          $(\mathcal{P}_\kappa(\mu_i))^V$.
      \end{enumerate}
  \end{enumerate}
  Then there is $\langle x_i \mid i < \theta \rangle \in W$ such that, for all $\alpha < (\mu^+)^V$ and all 
  sufficiently large $i < \theta$, we have $x_i \in \mathcal{C}_{\alpha, i}$.
\end{theorem}

\begin{proof}
  Let $\lambda = (\mu^+)^V$. In $V$, fix a sequence $\langle e_\beta \mid \beta
  < \lambda \rangle$ such that, for all 
  $\beta < \lambda$, $e_\beta: \beta \rightarrow \mu$ is an injection. As in the proof of Theorem \ref{diagonal_guessing_theorem},
  by shrinking the clubs $\mathcal{C}_{\alpha, i}$ if necessary, we may assume that:
  \begin{itemize}
    \item for all $i < \theta$, all $\alpha < \lambda$, and all $x \in \mathcal{C}_{\alpha, i}$, $y_i \subseteq x$;
    \item for all $\alpha < \beta < \lambda$ and all $i < \theta$ such that $e_\beta(\alpha) \in y_i \cap \mu_i$, we have 
      $\mathcal{C}_{\beta, i} \subseteq \mathcal{C}_{\alpha, i}$.
  \end{itemize}

  Since $\mu$ is strong limit in $V$, we have that, for all $i < \theta$, there are fewer than 
  $\mu$ clubs in $\mathcal{P}_\kappa(\mu_i)$ in $V$. Enumerate all such clubs as 
  $\langle \mathcal{D}_{i, \zeta} \mid \zeta < \delta_i \rangle$ for some $\delta_i < \mu$ (with the enumeration 
  being done in $V$). By assumption, 
  in $W$, $\delta_i = \bigcup_{k < \theta} (y_k \cap \delta_i)$. For each $k < \theta$, 
  fix $x_{i, k} \in \bigcap_{\zeta \in (y_k \cap \delta_i)} \mathcal{D}_{i, \zeta}$. 
  Then $\langle x_{i, k} \mid k < \theta \rangle$ has the property that, for every $\mathcal{C} \in V$ that is 
  club in $\mathcal{P}_\kappa(\mu_i)$, for all large enough $k < \theta$, $x_{i, k} \in \mathcal{C}$. 

  In $W$, for all $\alpha < \lambda$, fix a function $\sigma_\alpha:\theta \rightarrow \theta$ such that, 
  for all $i < \theta$, $x_{i, \sigma_\alpha(i)} \in \mathcal{C}_{\alpha, i}$. Since $\lambda > 2^\theta$, 
  there is an unbounded $A \subseteq \lambda$ and a fixed $\sigma$ such that, for all $\alpha \in A$, 
  $\sigma_\alpha = \sigma$. For $i < \theta$, let $x_i = x_{i, \sigma(i)}$. The verification that 
  $\langle x_i \mid i < \theta \rangle$ is as desired is exactly as in the proof of Theorem \ref{diagonal_guessing_theorem}.
\end{proof}

We end with the following corollary to Theorem \ref{theorem_58}. We thank the referee for 
pointing it out to us.

\begin{corollary}
  Suppose that $V$, $W$, $\kappa$, $\mu$, and $\langle \mu_i \mid i < \theta \rangle$ are as in the 
  statement of Theorem \ref{theorem_58} and $2^\mu = \mu^+$ in $W$. Suppose that, in $W$, 
  $\langle \pi_\alpha \mid \alpha < \mu^+ \rangle$ is an enumeration of $\prod_{i < \theta} 
  (\mathcal{P}_\kappa(\mu_i))^V$. Then, for all but boundedly many $i < \theta$, we have 
  $\langle \pi_\alpha(i) \mid \alpha < \mu^+ \rangle \not\in V$.
\end{corollary}

\begin{proof}
  Suppose not, and let $\langle \pi_\alpha \mid \alpha < \mu^+ \rangle$ be a counterexample. 
  Let $A$ be the set of $i < \theta$ such that $\langle \pi_\alpha(i) 
  \mid \alpha < \mu^+ \rangle \in V$. By assumption, $A$ is unbounded in $\theta$. 
  For $\alpha < \mu^+$ and $i \in A$, let $C_{\alpha, i} = \{x \in (\mathcal{P}_\kappa(\mu_i))^V 
  \mid \pi_\alpha(i) \subsetneq x\}$. For $\alpha < \mu^+$ and $i \in \theta \setminus A$, 
  let $C_{\alpha, i} = (\mathcal{P}_\kappa(\mu_i))^V$. Now $\langle C_{\alpha, i} \mid \alpha < \mu^+, i < \theta \rangle$ 
  satisfies Clause (7) of the statement of Theorem \ref{theorem_58}, so we can find a sequence 
  $\langle x_i \mid i < \theta \rangle$ as in the conclusion of Theorem \ref{theorem_58}. Then 
  there is $\alpha < \mu^+$ such that $x_i = \pi_\alpha(i)$ for all $i < \theta$. 
  Now, for all sufficiently large $i \in A$, we have $x_i \in C_{\alpha, i} = 
  \{x \in (\mathcal{P}_\kappa(\mu_i))^V \mid x_i \subsetneq x\}$, which is a contradiction.
\end{proof}

\bibliographystyle{amsplain}
\bibliography{pseudo-prikry}

\end{document}